\documentclass[11pt]{amsart}

\usepackage[T1]{fontenc}
\usepackage{lmodern}
\usepackage{amsmath}
\usepackage{amssymb}
\usepackage{mathtools}
\usepackage[hidelinks]{hyperref}

\newtheorem{theorem}{Theorem}[section]
\newtheorem{lemma}[theorem]{Lemma}
\newtheorem{proposition}[theorem]{Proposition}

\numberwithin{equation}{section}

\title[Average Local Independence and Leaf Number]
{Average Local Independence and the Spanning-Tree Leaf Number:\\
A Proof of Graffiti.pc Conjecture 2}

\author{Yanmohan Wang}
\author{Tianyue Dai}
\author{Rui Tong}

\date{}


\begin{document}

\begin{abstract}
We prove Graffiti.pc Conjecture~2, a 1996 conjecture listed as open
on the \emph{Written on the Wall II} page marked ``Last update
7/23/26.'' Let \(G\) be a finite simple connected graph. For
\(v\in V(G)\), let
\[
 I(v)=\alpha\bigl(G[N_G(v)]\bigr),
\]
and let \(I_{\mathrm{avg}}(G)\) be the average of these local
independence numbers. The conjecture states that the maximum number
\(L_s(G)\) of leaves in a spanning tree of \(G\) satisfies
\[
 L_s(G)\ge 2\bigl(I_{\mathrm{avg}}(G)-1\bigr).
\]
We establish this inequality by extracting a triangle-free spanning
subgraph that retains at
least half of the total local-independence mass. A degree-square
argument then produces a double star with sufficiently many leaves,
and this tree extends to a spanning tree without losing leaves.
Balanced complete bipartite graphs show that the bound is sharp.
\end{abstract}

\maketitle

\section{Introduction and main result}

All graphs in this note are finite and simple. For a graph \(G\) and
\(v\in V(G)\), write
\[
 N_G(v)=\{u\in V(G):uv\in E(G)\}
\]
for the open neighborhood of \(v\). Define the \emph{local
independence number} at \(v\) by
\[
 I(v)=\alpha\bigl(G[N_G(v)]\bigr),
\]
where \(\alpha(H)\) is the independence number of \(H\), with
\(\alpha(\varnothing)=0\). The average local independence number is
\[
 I_{\mathrm{avg}}(G)
 =\frac{1}{|V(G)|}\sum_{v\in V(G)} I(v).
\]

For a tree \(T\), let
\[
 \ell(T)=\bigl|\{v\in V(T):d_T(v)=1\}\bigr|.
\]
Thus a leaf means a vertex of degree exactly one; in particular, the
one-vertex tree has no leaves under this convention. If \(G\) is
connected, define its \emph{spanning-tree leaf number} by
\[
 L_s(G)=\max\{\ell(T):T\text{ is a spanning tree of }G\}.
\]

The inequality below is Graffiti.pc Conjecture~2, listed as item
O~2 on the open-problem page of \emph{Written on the Wall
II}~\cite{WOWII}. The page is marked ``Last update 7/23/26,''
assigns the item the year 1996, and lists it among the open
conjectures. In 2008, DeLaVi\~na and Waller recorded the same
inequality as Conjecture~4 and described it as ``another long
standing open conjecture of Graffiti''~\cite[Conjecture~4]{DW}.
Theorem~\ref{thm:main} establishes the stated inequality.

\begin{theorem}\label{thm:main}
Let \(G\) be a finite simple connected graph. Then
\[
 L_s(G)\ge
 2\left(
 \frac{1}{|V(G)|}\sum_{v\in V(G)}I(v)-1
 \right)
 =2\bigl(I_{\mathrm{avg}}(G)-1\bigr).
\]
\end{theorem}

The proof has two ingredients. First, the chosen independent
neighborhoods yield a relatively dense triangle-free spanning
subgraph. Second, an edge with large endpoint-degree sum in that
subgraph supports a double star with many leaves.

\section{Auxiliary lemmas}

\begin{lemma}[Tree extension]\label{lem:extension}
Let \(G\) be connected, and let \(F\subseteq G\) be a tree with at
least two vertices. Then \(F\) is contained in a spanning tree \(T\)
of \(G\) satisfying
\[
 \ell(T)\ge \ell(F).
\]
\end{lemma}

\begin{proof}
Start with \(F\). If the current tree has vertex set
\(U\subsetneq V(G)\), connectivity gives an edge \(ab\in E(G)\) with
\(a\in U\) and \(b\notin U\). Add \(b\) and the edge \(ab\). The
result remains a tree. If \(a\) was a leaf, then \(b\) replaces \(a\)
as a leaf; otherwise \(b\) is an additional leaf. Thus the number of
leaves never decreases. Repeating this step produces the required
spanning tree.
\end{proof}

\begin{lemma}[Triangle-free extraction]\label{lem:extraction}
Let \(G\) be a finite simple graph. For every \(v\in V(G)\), let
\(A_v\subseteq N_G(v)\) be independent. Then \(G\) has a
triangle-free spanning subgraph \(H\) such that
\[
 |E(H)|\ge \frac12\sum_{v\in V(G)}|A_v|.
\]
\end{lemma}

\begin{proof}
Fix a linear order \(\prec\) on \(V(G)\). Define a spanning subgraph
\(H_\prec\) by retaining an edge \(xy\in E(G)\) precisely when its
earlier endpoint selects its later endpoint. Explicitly,
\[
 xy\in E(H_\prec)
 \quad\Longleftrightarrow\quad
 \begin{cases}
 y\in A_x,&x\prec y,\\
 x\in A_y,&y\prec x.
 \end{cases}
\]

The graph \(H_\prec\) is triangle-free. Indeed, suppose that
\(x,y,z\) form a triangle in \(H_\prec\), and let \(x\) be the
earliest of the three vertices. The edges \(xy\) and \(xz\) then
force \(y,z\in A_x\). But \(yz\in E(H_\prec)\subseteq E(G)\), which
contradicts the independence of \(A_x\).

Let \(\prec^{\mathrm{rev}}\) denote the reverse order. For an edge
\(xy\in E(G)\) with \(x\prec y\), the exact indicator identity
\[
 \mathbf 1_{\{xy\in E(H_\prec)\}}
 +\mathbf 1_{\{xy\in E(H_{\prec^{\mathrm{rev}}})\}}
 =
 \mathbf 1_{\{y\in A_x\}}+\mathbf 1_{\{x\in A_y\}}
\]
holds. Summing over all edges gives
\[
 |E(H_\prec)|+|E(H_{\prec^{\mathrm{rev}}})|
 =\sum_{v\in V(G)}|A_v|.
\]
Both spanning subgraphs are triangle-free, so at least one of them
has the required number of edges.
\end{proof}

\begin{lemma}[Double-star bound]\label{lem:double-star}
Let \(G\) be a connected graph on \(n\) vertices, and let
\(H\subseteq G\) be a triangle-free spanning subgraph with
\(m=|E(H)|>0\). Then
\[
 L_s(G)\ge \frac{4m}{n}-2.
\]
\end{lemma}

\begin{proof}
Write \(d(v)=d_H(v)\). Double counting followed by the
Cauchy--Schwarz inequality gives
\[
 \sum_{xy\in E(H)}\bigl(d(x)+d(y)\bigr)
 =\sum_{v\in V(G)}d(v)^{2}
 \ge \frac{\left(\sum_v d(v)\right)^{2}}{n}
 =\frac{4m^2}{n}.
\]
Hence some edge \(xy\in E(H)\) satisfies
\[
 d(x)+d(y)\ge \frac{4m}{n}.
\]

Set
\[
 X=N_H(x)\setminus\{y\},
 \qquad
 Y=N_H(y)\setminus\{x\}.
\]
Since \(H\) is triangle-free, \(X\cap Y=\varnothing\). Let \(Q\) be
the subgraph with
\[
 V(Q)=\{x,y\}\cup X\cup Y
\]
and
\[
 E(Q)=\{xy\}\cup\{xu:u\in X\}\cup\{yv:v\in Y\}.
\]
This chosen subgraph is connected and has
\[
 |E(Q)|=1+|X|+|Y|=|V(Q)|-1,
\]
so \(Q\) is a tree. Every vertex of \(X\cup Y\) is a leaf of \(Q\).
Consequently,
\[
 \ell(Q)\ge |X|+|Y|=d(x)+d(y)-2
 \ge \frac{4m}{n}-2.
\]
If one of \(d(x),d(y)\) equals one, the corresponding center is an
additional leaf, so the displayed lower bound remains valid.
Lemma~\ref{lem:extension} now extends \(Q\) to a spanning tree of
\(G\) without decreasing its number of leaves.
\end{proof}

\section{Proof of the main theorem}

\begin{proof}[Proof of Theorem~\ref{thm:main}]
First suppose that \(n=|V(G)|\ge 2\), and put
\[
 S=\sum_{v\in V(G)}I(v).
\]
For each \(v\), choose an independent set
\[
 A_v\subseteq N_G(v),
 \qquad
 |A_v|=I(v).
\]
Lemma~\ref{lem:extraction} gives a triangle-free spanning subgraph
\(H\subseteq G\) with
\[
 m=|E(H)|\ge \frac S2.
\]
Because \(G\) is connected and \(n\ge2\), every vertex has a
neighbor, and hence \(I(v)\ge1\). Thus \(S\ge n\), so \(m>0\).
Lemma~\ref{lem:double-star} therefore applies and yields
\[
 \begin{aligned}
 L_s(G)
 &\ge \frac{4m}{n}-2\\
 &\ge \frac{2S}{n}-2\\
 &=2\left(
 \frac1n\sum_{v\in V(G)}I(v)-1
 \right).
 \end{aligned}
\]

If \(n=1\), then the unique vertex \(v\) has
\(N_G(v)=\varnothing\), so \(I(v)=0\). The unique spanning tree has
no degree-one vertices, and hence \(L_s(G)=0\). The asserted
inequality is then \(0\ge-2\).
\end{proof}

\section{Sharpness}

\begin{proposition}\label{prop:sharp}
For every integer \(r\ge2\), equality in Theorem~\ref{thm:main}
holds for \(K_{r,r}\). In particular, the coefficient \(2\) is best
possible.
\end{proposition}

\begin{proof}
Let \(L\) and \(R\) be the two parts of \(K_{r,r}\). The neighborhood
of every vertex is the opposite part, which is an independent set of
size \(r\). Therefore
\[
 I_{\mathrm{avg}}(K_{r,r})=r.
\]

Choose \(a\in L\) and \(b\in R\). The edges
\[
 ab,\qquad
 ay\quad(y\in R\setminus\{b\}),\qquad
 bx\quad(x\in L\setminus\{a\})
\]
form a spanning double star with \(2r-2\) leaves. Hence
\[
 L_s(K_{r,r})\ge2r-2.
\]

Conversely, let \(T\) be any spanning tree of \(K_{r,r}\). Since
each edge of \(T\) has exactly one endpoint in \(L\),
\[
 \sum_{x\in L}d_T(x)=|E(T)|=2r-1>r.
\]
Thus some vertex of \(L\) is not a leaf. The same argument for \(R\)
shows that \(R\) also contains a nonleaf. Hence \(T\) has at most
\(2r-2\) leaves. It follows that
\[
 L_s(K_{r,r})=2r-2
 =2\bigl(I_{\mathrm{avg}}(K_{r,r})-1\bigr).
\]
\end{proof}

\end{document}